\documentclass[12pt, reqno]{amsart}

\author[P.~Leonetti]{Paolo Leonetti}
\address{Institute of Analysis and Number Theory, Graz University of Technology,  Kopernikusgasse 24/II, 8010 Graz, Austria}
\email{leonetti.paolo@gmail.com}

\author[J.~Schwaiger]{Jens Schwaiger}
\address{Institute of Mathematics and Scientific Computing, University of Graz, $\text{ }$ Heinrichstra{\ss}e 36, 8010 Graz, Austria}
\email{jens.schwaiger@uni-graz.at}

\thanks{P.L. was supported by the Austrian Science Fund (FWF), project F5512-N26.}

\title{The general linear equation on open \\ connected sets}

\usepackage[T1]{fontenc}
\usepackage{amsmath}
\usepackage{amssymb}
\usepackage{amsthm}
\usepackage[left=3.5cm, right=3.5cm, paperheight=11.8in]{geometry}
\usepackage{hyperref}
\usepackage{fancyhdr}
\usepackage{enumitem}
\usepackage{comment}
\usepackage{nicefrac}
\usepackage{bm}
\usepackage{mathrsfs}
\usepackage{graphicx}
\usepackage[utf8]{inputenc}
\usepackage{cancel}
\usepackage{mathtools}

\newcommand{\vertiii}[1]{{\left\vert\kern-0.25ex\left\vert\kern-0.25ex\left\vert #1 
    \right\vert\kern-0.25ex\right\vert\kern-0.25ex\right\vert}}
		
\AtBeginDocument{%
   \def\MR#1{}
}

\newtheorem{theorem}{Theorem}[section]
\newtheorem{cor}[theorem]{Corollary}

\newtheorem{prop}[theorem]{Proposition}

\newtheorem{question}{Question}[section]
\theoremstyle{definition} 
\newtheorem{defi}[theorem]{Definition}
\let\olddefi\defi
\renewcommand{\defi}{\olddefi\normalfont}
\newtheorem{example}[theorem]{Example}
\let\oldexample\example
\renewcommand{\example}{\oldexample\normalfont}
\newtheorem{rmk}[theorem]{Remark}
\let\oldrmk\rmk
\renewcommand{\rmk}{\oldrmk\normalfont}

\newtheorem{claim}{\textsc{Claim}}
\newtheorem*{claim*}{\textsc{Claim}}

\providecommand{\RRb}{\mathbb{R}}

\hypersetup{
    pdftitle={{\textsc{fix}}},
    pdfauthor={fix},
    pdfmenubar=false,
    pdffitwindow=true,
    pdfstartview=FitH,
    colorlinks=true,
    linkcolor=blue,
    citecolor=green,
    urlcolor=cyan
}

\providecommand{\MR}[1]{}

\providecommand{\MR}{\relax\ifhmode\unskip\space\fi MR }

\providecommand{\href}[2]{#2}

\begin{document}

\maketitle
\thispagestyle{empty}

\maketitle 
\begin{abstract}
\noindent{} Fix non-zero reals $\alpha_1,\ldots,\alpha_n$ with $n\ge 2$ and let $K$ be a non-empty open connected set in a topological vector space such that $\sum_{i\le n}\alpha_iK\subseteq K$ (which holds, in particular, if $K$ is an open convex cone and $\alpha_1,\ldots,\alpha_n>0$). Let also $Y$ be a vector space over $\mathbb{F}:=\mathbb{Q}(\alpha_1,\ldots,\alpha_n)$. We show, among others, that a function $f: K\to Y$ satisfies the general linear equation
$$
\textstyle \forall x_1,\ldots,x_n \in K,\,\,\,\,\, f\left(\sum_{i\le n}\alpha_i x_i\right)=\sum_{i\le n}\alpha_i f(x_i)
$$
if and only if there exist a unique $\mathbb{F}$-linear $A:X\to Y$ and unique $b\in Y$ such that $f(x)=A(x)+b$ for all $x \in K$, with $b=0$ if $\sum_{i\le n}\alpha_i\neq 1$. 
The main tool of the proof is a general version of a result Rad\'{o} and Baker on the existence and uniqueness of extension of the solution on the classical Pexider equation.
\end{abstract}

\section{Introduction}\label{sec:intro}

Motivated by the study of certainty equivalents in the theory of decision making under uncertainty, the authors in \cite{MR3720973} solved a functional equation which, after some manipulations, led to the restricted general linear equation
\begin{equation}\label{eq:originalglmt}
\forall x,y \in \RRb^+ \times \RRb^+,\,\,\,\,\,\Phi\left(\alpha x+\beta y\right)=\alpha\,\Phi(x)+\beta\,\Phi(y),
\end{equation}
where $\alpha, \beta \in \RRb^+$ are given constants and $\Phi: \RRb^+ \times \RRb^+ \to \RRb^+$ is a continuous function such that $\lim_{x\to (0,0)}\Phi(x)=0$, with $\RRb^+:=(0,\infty)$, see \cite[Equation (6)]{MR3720973}.

The aim of this article is to provide a characterization of the solutions of general linear equations as in \eqref{eq:originalglmt}, where the variables are restricted to an open connected set of a topological vector space. The main novelty of the proof is a general version of a result Rad\'{o} and Baker  \cite{MR900703} on the existence and uniqueness of extension of the solution of the classical Pexider equation (see Theorem \ref{thm:rado} below). We refer to \cite[Chapter 13.10]{MR2467621} and \cite{MR0611544, Pal02, MR3359699} for the classical theory of general linear equations and references therein. 


Our main result follows. The proof is given in Section \ref{sec:proofjens}.
\begin{theorem}\label{thm:jensmain}
Let $X$ be a topological vector space over $\mathbb{K}$, where $\mathbb{K}$ is the field of real or complex numbers, and let $Y$ be a vector space over a field $\mathbb{F}$ of characteristic zero. Fix also non-zero $\alpha_1,\ldots,\alpha_n \in \mathbb{K}$ with $n\ge 2$, non-zero scalars $\beta_1,\ldots,\beta_n \in \mathbb{F}$, and a non-empty open connected set $K\subseteq X$ such that $\sum_{i\le n}\alpha_i K \subseteq K$. 
Finally, let $f: K \to Y$ be a function such that
\begin{equation}\label{eq:hypojens}
\textstyle \forall x_1,\ldots,x_n \in K,\,\,\,\,\,\,\,\,f\left(\sum_{i\le n}\alpha_ix_i\right)=\sum_{i\le n}\beta_i f(x_i).
\end{equation}

Then there exist a unique group homomorphism $A: X\to Y$ and a unique $b \in Y$ for which
\begin{equation}\label{eq:homogeneityA}
\forall x \in X, \forall i=1,\ldots,n, \,\,\,\,\,\,\,\,\,\,A(\alpha_i x)=\beta_i A(x)
\end{equation}
and
\begin{equation}\label{eq:claimjens}
\forall x \in K,\,\,\,\,\,\,\,\,\,\,f(x)=A(x)+b,
\end{equation}
where necessarily $b=0$ if $\sum_{i\le n}\beta_i \neq 1$ \textup{(}and $b \in Y$ arbitrary otherwise\textup{)}.

Conversely, if $A: X\to Y$ is a group homomorphism which satisfies \eqref{eq:homogeneityA}, and $f: K\to Y$ is defined by \eqref{eq:claimjens} with $b=0$ if $\sum_{i\le n}\beta_i \neq 1$, then $f$ satisfies \eqref{eq:hypojens}.
\end{theorem}


Related results can found, e.g., in \cite[Section 4]{MR0611544}. 
We remark that, in the case $X=K=\mathbb{R}^k$, $Y=\mathbb{R}$, and $n=2$, it is possible to characterize all group homomorphisms $A:X\to Y$ satisfying \eqref{eq:homogeneityA}, see e.g. \cite[Theorem 13.10.5]{MR2467621}. 

In the rather early paper \cite{MR0326217} we can find an extension theorem for the classical Cauchy equation restricted to an arbitrary non-empty open, not necessarily connected, subset of $\mathbb{R}^2$. Other investigations by the same authors related to exponential polynomials and spectral analysis can be found in \cite{Sze82, Sze89, Sze91}.



In some cases, the hypotheses of Theorem \ref{thm:jensmain} are sufficiently easy to check. In this regard, 
given a real topological vector space $X$, a set $K\subseteq X$ is said to be a \emph{convex cone} 
if $K+K\subseteq K$ and $\alpha K \subseteq K$ 
for all real $\alpha>0$. Note that a convex cone $K$ is connected and that $\alpha K=K+K=K$ for all $\alpha>0$. Therefore:
\begin{cor}\label{cor:cones}
Theorem \ref{thm:jensmain} holds if $K$ is a non-empty open convex cone and $\alpha_1,\ldots,\alpha_n>0$, provided that $\mathbb{K}=\mathbb{R}$.
\end{cor}

As another consequence, we have the following.
\begin{cor}\label{thm:jensmain2222}
With the same hypothesis of Theorem \ref{thm:jensmain}, let us suppose that $\mathbb{K}=\mathbb{R}$, $\mathbb{F}:=\mathbb{Q}(\alpha_1,\ldots,\alpha_n)$, and $\alpha_i=\beta_i$ for all $i=1,\ldots,n$. 

Then a function $f: K \to Y$ satisfies
\begin{equation}\label{eq:hypojens2222}
\textstyle \forall x_1,\ldots,x_n \in K,\,\,\,\,\,\,\,\,f\left(\sum_{i\le n}\alpha_ix_i\right)=\sum_{i\le n}\alpha_i f(x_i).
\end{equation}
if and only if there exist a $\mathbb{F}$-linear $A: X\to Y$ and $b \in Y$ such that $f(x)=A(x)+b$ for all $ x \in K$, with $b=0$ if $\sum_{i\le n}\alpha_i\neq 1$.
\end{cor}
\begin{proof}
Thanks to Theorem \ref{thm:jensmain}, we just need to show that if $f$ satisfies \eqref{eq:hypojens2222} then $A$ is $\mathbb{F}$-linear. We obtained that $A$ is $\alpha_i$-homogeneous for all $i=1,\ldots,n$, i.e., $A(\alpha_i x)=\alpha_i A(x)$ for all $x \in X$. By a straightforward argument, $A$ is $p(\alpha_1,\ldots,\alpha_n)$-homogeneous, for each polynomial $p \in \mathbb{Q}[X_1,\ldots,X_n]$, hence also to the corresponding field of rational functions. 
\end{proof}


In particular, 
we obtain all the solutions of Equation \eqref{eq:originalglmt} (we omit details).
\begin{cor}\label{cor:glmt}
Set $K:=\RRb^+\times \RRb^+$ and fix $\alpha,\beta \in \RRb^+$. Then a continuous function $\Phi: K \to \RRb^+$ such that $\lim_{x\to (0,0)}\Phi(x)=0$ satisfies Equation \eqref{eq:originalglmt} if and only if there exists $y \in \RRb^2$ such that $\Phi(x)=\langle x, y\rangle$ for all $x \in K$.
%
%
\end{cor}

Lastly, thanks to Theorem \ref{thm:jensmain}, we recover the characterization of the solutions of Jensen-like equations (again, we omit details). This may be used to shorten the proof of the main result in \cite{MR3542948}, cf. Equation (11) and (12) therein. 
\begin{cor}\label{cor:matkowskioriginal}
Fix positive reals $\alpha_1,\ldots,\alpha_n$ with $n\ge 2$ and $\sum_{i\le n}\alpha_i= 1$. 
Then a continuous function $\Phi: [0,1]^k\to \mathbb{R}^h$ satisfies Equation \eqref{eq:hypojens2222} if and only if there exist a $h\times k$ real matrix $A$ and $b \in \mathbb{R}^h$ such that $\Phi(x)=Ax+b$ for all $x \in [0,1]^k$.
\end{cor}


As we anticipated before, 
the main novelty in the proof of Theorem \ref{thm:jensmain} is the use of (a natural extension of) a result of Rad\'{o} and Baker \cite{MR900703} on the existence and uniqueness of extension of the solution of the classical Pedixer's equation. 

To this aim, we need to fix some notation. Given non-empty sets $A,B$ and $n \in \mathbb{N}^+$ with $A\subseteq B^n$, denote by $\pi_i$ the $i$-th projection, that is, $\pi_i(a):=a_i$ for all $a=(a_1,\ldots,a_n) \in A$, and define
$$
A_i:=\pi_i(A)=\{b \in B:  \pi_i(a)=b\,\text{ for some } a \in A\}
$$
for each $i=1,\ldots,n$. Finally, set 
$$
A_+:=\{a_1+\cdots+a_n: (a_1,\ldots,a_n) \in A\}.
$$

\begin{theorem}\label{thm:rado}
\textbf{\textup{(}Rad\'{o} and Baker's extension theorem.\textup{)}} Let $X$ be a topological vector space over $\mathbb{K}$, where $\mathbb{K}$ is the field of real or complex numbers, and $Y$ be an abelian group. Moreover, given $n\ge 2$, let $U\subseteq X^n$ be a non-empty open connected set and fix functions $f: U_+ \to Y$, $g_1: U_1\to Y, \ldots, g_n: U_n \to Y$ such that
\begin{equation}\label{eq:pexidertwo}
\textstyle \forall \,(x_1,\ldots,x_n) \in U,\,\,\,\,\,\,\,f\left(\sum_{i\le n}x_i\right)=\sum_{i\le n}g_i(x_i).
\end{equation}
Then there exists a unique extension $(F,G_1,\ldots,G_n)$ of $(f,g_1,\ldots,g_n)$ for which
$$
\textstyle \forall\, x_1,\ldots,x_n \in X,\,\,\,\,\,\,\,F\left(\sum_{i\le n}x_i\right)=\sum_{i\le n}G_i(x_i).
$$
In fact, there exist a unique group homomorphism $A: X \to Y$ and unique $u,u_1,\ldots,u_n \in Y$ such that $u=\sum_{i\le n}u_i$ and 
$$
\forall x \in X,\forall i=1,\ldots,n,\,\,\,\,\,\,\,F(x)=A(x)+u\,\,\text{ and }\,\, \,G_i(x)=A(x)+u_i.
$$
\end{theorem}
The proof for the case $n=2$ can be found in \cite[Theorem 1]{MR900703}, cf. also \cite[Theorem 4, p.80]{MR875412}. Moreover, see \cite[Theorem 5]{MR1130842} and \cite{MR2161144, ChuTab08, MR3648479} for related results. The proof of Theorem \ref{thm:rado} follows in Section \ref{sec:proof}.

Lastly, if the Pexider equation \eqref{eq:pexidertwo} holds for all $x_1,\ldots,x_n\in X$, then its analogue holds in a more general context; see \cite[Proposition 1]{GKC2019} for a related result.
\begin{prop}\label{thm:simple}
Let $X,Y$ be abelian groups, written additively, and fix functions $f, g_1,\ldots,g_n: X\to Y$, with $n\ge 2$. Then
\begin{equation}\label{eq:mainfunctional}
\textstyle \forall x_1,\ldots,x_n \in X,\,\,\,\,\,\,f\left(\sum_{i\le n}x_i\right)=\sum_{i\le n}g_i(x_i)
\end{equation}
if and only if there exist 
a homomorphism $A: X\to Y$ and $y,y_1,\ldots y_n \in Y$ with $y=\sum_{i\le n}y_i$ such that 
$$
\forall x\in X, \forall i=1,\ldots,n,\,\,\,\,\,f(x)=A(x)+y\,\,\,\,\text{ and }\,\,\,\,g_i(x)=A(x)+y_i.
$$
\end{prop}
\begin{proof}
The \textsc{if} part is clear. Conversely, define $y_i:=g_i(0)$ for all $i=1,\ldots,n$, where $0$ is the identity of $X$, and note that, setting $x_1=\cdots=x_n=0$ in \eqref{eq:mainfunctional}, we obtain 
$
f(0)=y:=\sum_{i\le n} y_i.
$ 
Hence, define the functions $\tilde{f}, \tilde{g}_1,\ldots,\tilde{g}_n: X\to Y$ by $\tilde{f}(x):=f(x)-y$ and $\tilde{g}_i(x)=g_i(x)-y_i$ for all $x \in X$ and $i=1,\ldots,n$. It follows by \eqref{eq:mainfunctional} that
\begin{equation}\label{eq:mainfunctional2}
\textstyle \tilde{f}\left(\sum_{i\le n} x_i\right)=\sum_{i\le n} \tilde{g}_i(x_i)
\end{equation}
for all $x_1,\ldots,x_n \in X$, and by construction $\tilde{f}(0)=\tilde{g}_1(0)=\cdots=\tilde{g}_n(0)=0$, where the latter $0$ is the identity of $Y$. 
Given $x \in X$ and $i \in \{1,\ldots,n\}$, set $x_i=x$ and $x_j=0$ for all $j \in \{1,\ldots,n\}\setminus \{i\}$ in \eqref{eq:mainfunctional2} so that 
$
\tilde{f}(x)=\tilde{g}_i(x).
$ 
By the arbitrariness of $x$, we conclude that $\tilde{f}=\tilde{g}_i$, hence 
$
\tilde{f}\left(\sum_{i\le n} x_i\right)=\sum_{i\le n} \tilde{f}(x_i)
$ 
for all $x_1,\ldots,x_n \in X$. Setting $x_1=\cdots=x_{n-2}=0$, 
we see that $\tilde{f}$ itself is a homomorphism. Therefore $A=\tilde{f}$, so that $f(x)=A(x)+y$ and $g_i(x)=A(x)+y_i$ for all $i=1,\ldots,n$ and $x \in X$.
\end{proof}

However, the analogous statement for the functional equation 
\begin{equation}\label{eq:mainfunctionalweight}
\textstyle \forall x_1,\ldots,x_n \in X,\,\,\,\,\,\,f\left(\sum_{i\le n}\alpha_ix_i\right)=\sum_{i\le n}g_i(x_i)
\end{equation}
where $\alpha_1,\ldots,\alpha_n$ are fixed non-zero integers does not hold. Indeed, consider the following example.
\begin{example}
Set $n=2$, $X=Y=\mathbb{Z}_4$, $\alpha_1=\alpha_2=2$, and fix functions $f,g_1,g_2: X\to Y$ such that $f(0)=f(2)=0$, $f(1)=1$, and $g_1(x)=g_2(x)=0$ for all $x \in X$. Then the functional equation \eqref{eq:mainfunctionalweight} holds.  
However, if there exist a homomorphism $A:X\to Y$ and $y\in Y$ such that $f(x)=A(x)+y$ for all $x \in X$, then
$$
3=f(2)-f(1)=(A(2)+y)-(A(1)+y)=A(1),
$$
therefore $A(2)=2$, $y=f(1)-A(1)=2$, and $A(0)=f(0)-y=2$. This is impossible since we should have $A(0)=0$.
\end{example}


\section{Proof of Theorem \ref{thm:rado}}\label{sec:proof}

Fix $x=(x_1,\ldots,x_n) \in U$. Since $U$ is open, there exists a neighborhood $G_x$ of $0 \in X^n$ such that $x+G_x \subseteq U$. In particular, by the standing assumptions, we have 
$f\left(\sum_{i\le n}(x_i+y_i)\right)=\sum_{i\le n}g_i(x_i+y_i)$ for all $(y_1,\ldots,y_n) \in G_x$. 
Since each projection $\pi_i$ is an open map, it follows that $V_x:=\bigcap_{i\le n} \pi_i(G_x)$ is a non-empty open neighborhood of $0 \in X$. In particular, 
$$
\textstyle \forall \, z_1,\ldots,z_n \in V_x,\,\,\,\,f\left(\sum_{i\le n}(x_i+z_i)\right)=\sum_{i\le n}g_i(x_i+z_i).
$$
At this point, define the functions $\tilde{f}, \tilde{g}_1,\ldots,\tilde{g}_n: V_x \to Y$ by 
$$
\textstyle \tilde{f}(z):=f\left(z+\sum_{i\le n}x_i\right)-f\left(\sum_{i\le n}x_i\right)\,\,\,\text{ and }\,\,\,\tilde{g}_i(z):=g_i(x_i+z)-g_i(x_i)
$$
for each $z \in V_x$ and $i=1,\ldots,n$. Since $f\left(\sum_{i\le n}x_i\right)=\sum_{i\le n}g_i(x_i)$, it follows that
$$
\textstyle \forall \, z_1,\ldots,z_n \in V_x,\,\,\,\,\tilde{f}\left(\sum_{i\le n}z_i\right)=\sum_{i\le n}\tilde{g}_i(z_i),
$$
and, in addition, $\tilde{f}(0)=\tilde{g}_1(0)=\cdots=\tilde{g}_n(0)=0$. 
As in the proof of Theorem \ref{thm:simple}, we obtain that  
$$
\textstyle \forall z_1,z_2 \in V_x,\,\,\,\,\,\tilde{f}(z_1+z_2)=\tilde{g}_1(z_1)+\tilde{g}_2(z_2)
$$ 
and the restriction $\tilde{f}$ to $ V_x$ coincides with each of $\tilde{g}_1,\ldots,\tilde{g}_n$. 
Moreover, as in the proof of \cite[Theorem 1]{MR900703}, there exists a unique group homomorphism $A_x: X\to Y$ which extends $\tilde{f}$. 

To sum up, this implies that, for each $x \in U$,  there exist a neighborhood $V_x$ of $0 \in X$, a unique group homomorphism $A_x: X \to Y$, and unique $u_x,u_{x,1},\ldots,u_{x,n} \in Y$ with $u_x=\sum_{i\le n}u_{x,i}$ such that
$$
\textstyle \forall z \in V_x, \forall i=1,\ldots,n,\,\,\,\,g_i(x_i+z)=A_x(z)+u_{x,i}
$$
and
$$
\textstyle \forall z \in n\cdot V_x,\,\,\,\,\,f\left(z+\sum_{i\le n}x_i\right)=A_x(z)+u_x,
$$
where $n\cdot V_x:=V_x+\cdots+V_x$, where $V_x$ is repeated $n$ times.

\begin{claim}\label{connect}
Fix $x,y \in U$ such that $(x+V_x^n) \cap (y+V_y^n) \neq \emptyset$. Then $A_x=A_y$.
\end{claim}
\begin{proof}
For each $i=1,\ldots,n$, we get by hypothesis that $(x_i+V_x) \cap (y_i+V_y)$ is a non-empty open set in $X$, hence there exists $w_i \in X$ and a neighborhood $W_i$ of $0 \in X$ such that $w_i+W_i \subseteq (x_i+V_x) \cap (y_i+V_y)$. Set $w=(w_1,\ldots,w_n)$ and note that 
$
\textstyle Z:=V_w \cap W_1 \cap \cdots \cap W_n
$ 
is a non-empty open neighborhood of $0 \in X$. Hence
$$
w+Z^n \subseteq (x+V_x^n) \cap (y+V_y^n).
$$
For each $z \in Z$, we have $w_i+z \in x_i+V_x$ for each $i=1,\ldots,n$, so that
\begin{displaymath}
\begin{split}
A_w(z)&=g_i(w_i+z)-u_{w,i}=g_i(x_i+(w_i-x_i+z))-u_{w,i}\\
&=u_{x,i}+A_x(w_i-x_i+z)-u_{w,i}=(u_{x,i}-u_{w,i}+A_x(w_i-x_i))+A_x(z).
\end{split}
\end{displaymath}
Setting $z=0$ we obtain 
\begin{equation}\label{eq:constants}
u_{x,i}-u_{w,i}+A_x(w_i-x_i)=0,
\end{equation}
therefore $A_w(z)=A_x(z)$ for all $z \in Z$. Considering that $\bigcup_{n\ge 1}nZ=X$ by \cite[Theorem 1.15.a]{MR1157815} and both $A_x$ and $A_w$ are $\mathbb{Q}$-linear, we obtain that $A_w(z)=A_x(z)$ for all $z \in X$. Similarly $A_w(z)=A_y(z)$ for all $z \in X$, therefore $A_x=A_y$.
\end{proof}

\begin{claim}\label{intermediate}
With the same hypothesis of Claim \ref{connect}, there exist unique $u,u_1,\ldots,u_n \in Y$ with $u=\sum_{i\le n}u_i$ such that
$$
\textstyle \forall i=1,\ldots,n,\forall z \in (x_i+V_x)\cup (y_i+V_y),\,\,\,\,\,\,\,g_i(z)=A_x(z)+u_i
$$
and
$$
\textstyle \forall z \in \left(\sum_{i\le n}x_i+n\cdot V_x\right)\cup \left(\sum_{i\le n}y_i+n\cdot V_y\right) ,\,\,\,\,\,f(z)=A_x(z)+u.
$$
\end{claim}
\begin{proof}
Set $A:=A_x=A_y$. With the same notation of the proof of Claim \ref{connect}, as it follows from Equation \eqref{eq:constants}, we have
$$
u_i:=u_{x,i}-A(x_i)=u_{w,i}-A(w_i)=u_{y,i}-A(y_i)
$$
for each $i=1,\ldots,n$. Hence, for each $z \in x_i+V_x$, it holds
$$
g_i(z)=g_i(x_i+z-x_i)=A(z-x_i)+u_{x,i}=
A(z)+u_i,
$$
and similarly for $z \in y_i+V_y$. 

To conclude, for each $z \in \sum_{i\le n}x_i+n\cdot V_x$, there exist $v_1,\ldots,v_n \in V_x$ such that $z=\sum_{i\le n}(x_i+v_i)$, hence
\begin{displaymath}
\begin{split}
\textstyle f(z)&\textstyle=f\left(\sum_{i\le n}(x_i+v_i)\right)=\sum_{i\le n}g_i(x_i+v_i)=\sum_{i\le n}(A(v_i)+u_{x,i})\\
&\textstyle =\sum_{i\le n}(A(v_i)+A(x_i)+u_{i})=A\left(\sum_{i\le n}(x_i+v_i)\right)+\sum_{i\le n}u_i=A(z)+u,
\end{split}
\end{displaymath}
and similarly for $z \in \sum_{i\le n}y_i+n\cdot V_y$.
\end{proof}

\begin{claim}\label{cuhd}
Fix $x,y \in U$. Then there exists a finite sequence $t_0,t_1,\ldots,t_k \in U$ such that $t_0=x$, $t_k=y$, and $(t_{j}+V_{t_j}^n) \cap (t_{j+1}+V_{t_{j+1}}^n) \neq \emptyset$ for each $j=0,1,\ldots,k-1$.
\end{claim}
\begin{proof}
Note that $\{x+V_x^n: x \in U\}$ is a family of open sets in $U$ such that $\bigcup_{x \in U}(x+V_x^n)=U$. Then, the claim follows by \cite[Lemma 2.4]{MR3648479}.
\end{proof}

Putting together Claims \ref{connect}, \ref{intermediate}, and \ref{cuhd}, we conclude the proof.

\section{Proof of Theorem \ref{thm:jensmain}}\label{sec:proofjens}

First, let us assume that $f: K\to Y$ is a function satisfying \eqref{eq:hypojens}. For each $i=1,\ldots,n$, define the function $g_i: K\to Y$ by $g_i(x)=\beta_i f(x/\alpha_i)$ for all $x \in K$. It follows that 
\begin{equation}\label{eq:additivepexider1}
\textstyle \forall x_1,\ldots,x_n\in K,\,\,\,\,\,\,f\left(\sum_{i\le n}x_i\right)=\sum_{i\le n}g_i(x_i).
\end{equation}

Setting $U:=\prod_{i\le n}\alpha_iK$, it is readily seen that $U$ is a non-empty open connected subset of $X^n$ (with the usual product topology). Moreover, by construction $U_i=\alpha_i K$ for all $i=1,\ldots,n$ and $U_+=\sum_{i\le n}\alpha_i K$. 
Thanks to \eqref{eq:additivepexider1} and the hypothesis $\sum_{i\le n}\alpha_i K\subseteq K$, we obtain 
$$
\textstyle \forall x \in U,\,\,\,\,\,\,f\left(\sum_{i\le n}x_i\right)=\sum_{i\le n}g_i(x_i),
$$
hence by Theorem \ref{thm:rado} there exist a unique group homomorphism $A: X \to Y$ (hence, $A$ is $\mathbb{Q}$-linear) and unique $u,u_1,\ldots,u_n\in Y$ such that
\begin{equation}\label{eq:system1}
\forall x \in K,\forall i=1,\ldots,n,\,\,\,\,\,f(x)=A(x)+u,\,\,\,g_i(x)=A(x)+u_i,
\end{equation}
with $u=\sum_{i\le n}u_i$.

At this point, we claim that $A$ satisfies \eqref{eq:homogeneityA}. 
To this aim, fix $i \in \{1,\ldots,n\}$. 
Taking into account \eqref{eq:system1} and the definition of $g_i$, we get
\begin{equation}\label{eq:system233}
\forall x \in K,\,\,\,\,\,A(x)+u=f(x)=\frac{g_i(\alpha_i x)}{\beta_i}=\frac{A(\alpha_i x)+u_i}{\beta_i}.
\end{equation}
However, since $A$ is $\mathbb{Q}$-linear, then also
$$
\forall x \in K,\,\,\,\,\, 2A(x)+u=\frac{2A(\alpha_i x)+u_i}{\beta_i}.
$$
Calculating the differences of the above equations, we obtain that $A(\alpha_i x)=\beta_i A(x)$ for all $x \in K$. To conclude, fix $x \in X$, $y \in K$, and let $V$ be a neighborhood of $0$ such that $y+V \subseteq K$ (which exists since $K$ is open). It follows by \cite[Theorem 1.15.a]{MR1157815} that there exists $n \in \mathbb{N}$ such that $x \in nV$. Therefore, considering that $A$ is $\mathbb{Q}$-linear, we obtain 
$$
A(\alpha_i y)+\frac{A\left(\alpha_i x\right)}{n}=A\left(\alpha_i\left(y+\frac{x}{n}\right)\right)=\beta_i A\left(y+\frac{x}{n}\right)=\beta_i A(y)+\frac{\beta_i A(x)}{n},
$$
so that $A(\alpha_i x)=\beta_i A(x)$ for all $x \in X$.

Note that, thanks to \eqref{eq:system233}, we have $u_i=\beta_i u$ for each $i=1,\ldots,n$. Summing these equations we obtain 
$
u=\sum_{i\le n}u_i=u\sum_{i\le n}\beta_i.
$ 
Considering that $\mathbb{F}$ has characteristic $0$, it follows that $u=u_1=\cdots=u_n=0$ is the unique solution if $\sum_{i\le n}\beta_i \neq 1$; finally, $u$ can be any value in $Y$ if $\sum_{i\le n}\beta_i=1$. This concludes the proof of the first part.

\medskip

Conversely, let us assume that $A: X\to Y$ is a group homomorphism which satisfies \eqref{eq:homogeneityA} and $f: K\to Y$ is a function defined by \eqref{eq:claimjens}. If $\sum_{i\le n}\beta_i\neq 1$ and $b=0$ then $f(x)=A(x)$ 
so that 
$$
\textstyle f\left(\sum_{i\le n}\alpha_ix_i\right)=A\left(\sum_{i\le n}\alpha_ix_i\right)=\sum_{i\le n}\beta_i A(x_i)=\sum_{i\le n}\beta_i f(x_i)
$$
for all $x_1,\ldots,x_n \in K$. On the other hand, if $\sum_{i\le n}\beta_i=1$ and $b \in Y$ then
\begin{displaymath}
\textstyle f\left(\sum_{i\le n}\alpha_ix_i\right)=b+\sum_{i\le n}\beta_i A(x_i)=\sum_{i\le n}\beta_i\left(A(x_i)+b\right)=\sum_{i\le n}\beta_i f(x_i).
\end{displaymath}
Therefore, in both cases, $f$ satisfies \eqref{eq:hypojens}.


\section{Concluding Remark}

In our Theorem \ref{thm:jensmain} one could choose, e.g., $K=(-1,2)$ and fix non-zero reals $\alpha_1,\ldots,\alpha_n$ with $\sum_{i\le n}|\alpha_i|<1/2$ so that, in fact, $\sum_{i\le n}\alpha_i K \subseteq (-1,1)\subseteq K$ (the same works, for instance, if $X=\mathbb{C}$, $K$ is the open circle with center $\nicefrac{1}{2}$ and radius $\nicefrac{3}{2}$, and $\alpha_1,\ldots,\alpha_n$ are non-zero complex numbers such that $\sum_{i\le n}|\alpha_i|<1/2$).

However, 
if $X$ is real topological vector space, 
one may ask whether it would be sufficient to require that $\alpha_1,\ldots,\alpha_n$    
 are \emph{positive} reals. 
More precisely: 
\begin{question}\label{questionfinal}
Let $X$ be a real topological vector space, fix non-zero reals $\alpha_1,\ldots,\alpha_n$, and let $K\subseteq X$ be a non-empty open connected set with $\sum_{i\le n}\alpha_iK\subseteq K$. 
Does there exist a non-empty open connected set $K^\prime\subseteq K$ such that $\sum_{i\le n}|\alpha_i|K^\prime \subseteq K^\prime$? 
\end{question}


We can show that the answer is affirmative if $X=\mathbb{R}$ and $\sum_{i\le n}|\alpha_i| \le 1$. In such case, indeed, it would be sufficient to prove that $K^\prime:= K\cap (-K)$ is a non-empty neighborhood of $0$ so that $\sum_{i\le n}|\alpha_i|K^\prime$ is contained in both $\sum_{i\le n}\alpha_i K\subseteq K$ and $-\sum_{i\le n}\alpha_iK\subseteq -K$, hence also to the intersection $K^\prime$.

To this aim, let us assume that at least one $\alpha_i$ is negative, let us say $\alpha_1,\ldots,\alpha_k<0$ and $\alpha_{k+1},\ldots,\alpha_n>0$, for some positive integer $k\le n$. Define also $\alpha^+:=\sum_{i>k}\alpha_i$ and $\alpha^-:=\sum_{i\le k}\alpha_i$ so that $\alpha^-\neq 0$, $\alpha^+\neq 1$, and $\alpha^++\alpha^-\le \alpha^+-\alpha^-\le 1$.  
Since $K$ is a non-empty open connected set, there exist $a,b \in \mathbb{R}\cup \{\pm \infty\}$, with $a<b$, such that $K=(a,b)$. 
Note that $\alpha^- x+\alpha^+y \in K$ for all $x,y \in K$. Thus, if $a=-\infty$ then, given any $y_0 \in K$ we have $\alpha^- (-n)+\alpha^+y_0 \in K$ for all sufficiently large $n$, which implies $b=\infty$; thus $K=\mathbb{R}$. 
The case $b=\infty$ is similar. 
%

Hence, let us assume hereafter that $a$ and $b$ are finite. Note that $\alpha K=(\alpha a, \alpha b)$ and $-\alpha K=(-\alpha b, -\alpha a)$ for all $\alpha>0$; in addition, $(x,y)+(x^\prime,y^\prime)=(x+x^\prime,y+y^\prime)$ for all non-empty intervals $(x,y)$ and $(x^\prime,y^\prime)$. Therefore
$$
\sum_{i\le n}\alpha_iK=\alpha^+ K+\alpha^-K=(\alpha^+ a+\alpha^- b, \alpha^+ b+\alpha^- a).
$$
Considering that $\sum_{i\le n}\alpha_i K \subseteq K$, we obtain that $\alpha^+ a+\alpha^- b \ge a$ and $\alpha^+ b+\alpha^- a\le b$, which can be rewritten as 
\begin{equation}\label{eq:finalsysteminequalities}
(1-\alpha^+)a-\alpha^- b \le 0\,\,\,\,\,\text{ and }\,\,\,\,\,\, (1-\alpha^+) b-\alpha^- a\ge 0.
\end{equation}
If $\alpha^+-\alpha^-=1$ then, by \eqref{eq:finalsysteminequalities}, $-\alpha^-(a+b)=0$, so that $a\le 0$ and $b=-a\ge 0$ (and they cannot be equal since $(a,b)\neq \emptyset$). 
Otherwise $\alpha^+-\alpha^-<1$ and, in particular, $\alpha^++\alpha^-<1$. 
Multiplying the second equation 
in \eqref{eq:finalsysteminequalities} 
by $\frac{1-\alpha^+}{\alpha^-}<0$ and summing it to the first one, we obtain
$$
\left(-\alpha^-+\frac{(1-\alpha^+)^2}{\alpha^-}\right)b=
\frac{(1-\alpha^++\alpha^-)(1-\alpha^+-\alpha^-)}{\alpha^-}\,b\le 0.
$$
This implies that $b\ge 0$ and, similarly, $a\le 0$. 

To conclude, we claim that $0 \in (a,b)$. 
Let us assume for the sake of contradiction that $a=0$. Then, choosing $x_1=\cdots=x_k=b-\varepsilon \in K$ and $x_{k+1}=\cdots=x_n=\varepsilon \in K$ with $\varepsilon>0$ sufficiently small, 
we obtain that $\sum_{i\le n}\alpha x_i=\alpha^-(b-\varepsilon)+\alpha^+\varepsilon>0$, which is impossible. The case $b=0$ is similar.

\subsection{Acknowledgments}

The authors are grateful to an anonymous referee for suggestions that helped improving the overall presentation of the article.

\bibliographystyle{amsplain}
\bibliography{functional}

\end{document}